\documentclass[11pt]{article}
\usepackage{latexsym}
\usepackage{amsthm}
\usepackage{amssymb}
\usepackage{amsmath}
\usepackage{rotating}
\usepackage{multirow}
\usepackage{mathrsfs}
\usepackage{wasysym}
\usepackage{enumerate}
\usepackage{esint}
\usepackage{hyperref}
\usepackage{rawfonts}
\input{prepictex}
\input{pictex}
\input{postpictex}
\DeclareMathOperator{\End}{End}

\usepackage[OT2,OT1]{fontenc}
\def\cyr{%
\renewcommand\rmdefault{wncyr}%
\renewcommand\sfdefault{wncyss}%
\renewcommand\encodingdefault{OT2}%
\normalfont
\selectfont}
\DeclareMathAlphabet{\zap}{OT1}{pzc}{m}{it}
\DeclareTextFontCommand{\textcyr}{\cyr}
\def\be{\begin{equation}}
\def\ee{\end{equation}}
\def\bea{\begin{eqnarray*}}
\def\eea{\end{eqnarray*}}

\newcommand{\rad}{\text{\cyr   ya}}

\newtheorem{main}{Theorem}

\DeclareMathOperator{\Hess}{Hess}

\DeclareMathOperator{\Iso}{Iso}

\DeclareMathOperator{\trace}{Trace}

\DeclareMathOperator{\espan}{span}

\newtheorem{thm}{Theorem}
\newtheorem{lem}{Lemma}
\newtheorem{prop}{Proposition}
\newtheorem{cor}{Corollary}
\newtheorem{defn}{Definition}

\newenvironment{rmk}{\mbox{ }\\{\bf  Remark.}\mbox{ }}{
\hfill $\diamondsuit$\mbox{}\bigskip}
\def\ZZ{{\mathbb Z}}
\def\RR{{\mathbb R}}
\def\CP{{\mathbb C \mathbb P}}
\def\dist{\text{\sf dist}}

\def\rad{\text{\cyr ya}}
\begin{document}

\title{Gravitational Instantons, Weyl Curvature,\\ and Conformally K\"ahler Geometry}

\author{\parbox{1.5in}{\center Olivier Biquard,\\
{\scriptsize   Sorbonne Universit\'e}} \parbox{1.5in}{\center Paul Gauduchon, \\
{\scriptsize   \'Ecole Polytechnique }}  and \parbox{1.75in}{\center Claude LeBrun\thanks{Supported in part by   NSF grant DMS-2203572.}\\
{\scriptsize  Stony Brook University} }}

\date{}
\maketitle

\section{Introduction}

In the late 1970s, Stephen Hawking \cite{hawking} and Gary Gibbons \cite{gibhawk}, along with a small group of
other gravitational physicists at Cambridge, first  began their    systematic exploration of the multiverse  of 
complete, non-compact, Ricci-flat Riemannian $4$-manifolds. They  termed  such spaces  {\sf gravitational instantons}, 
in the expectation that these  would eventually come to represent  tunneling modes in    a   theory  of quantum  gravity.  
Their striking discoveries included the construction of two infinite families of half-conformally-flat
 gravitational  instantons,  respectively generalizing
   the 
Eguchi-Hanson metric \cite{EH} and the Euclidean Taub-NUT metric, 
and specific properties of these unanticipated families turned out to have
long-term implications for the direction of differential-geometric research. Indeed, the fact that the Ricci-flat metrics in these  new families
were all anti-self-dual  then  allowed Hitchin \cite{hitpoly} to pioneer  an essentially independent approach to them, 
based on Penrose's nonlinear graviton construction \cite{pnlg}. Soon afterwards, it was then  realized these metrics were also {\sf hyper-K\"ahler}, in the 
new sense that had   recently been introduced  by Calabi \cite{calhk}, and considering them in this broader context   not only gave rise to new 
methods \cite{HKLR,kronquot} of constructing such spaces, but also led to  a satisfyingly complete classification of 
asymptotically locally Euclidean (ALE)  hyper-K\"ahler gravitational instantons \cite{krontor}. 
This  inspired  repeated  flurries of  intensive mathematical activity \cite{chenchen,cherhit,hein,HSVZ,minerbe,szhk}
during the following decades,  eventually resulting  in  an essentially complete  classification 
of hyper-K\"ahler gravitational instantons 
with curvature  in $L^2$. (However, since there are known  examples \cite{AKL} of complete hyper-K\"ahler $4$-manifolds whose
curvature is {\em not} in $L^2$,  the current  advanced state of the subject may still not represent the end of the story.) 
Such results also allow one to deduce   classification theorems for complete Ricci-flat $4$-manifolds  
that are  K\"ahler \cite{ioana2} or anti-self-dual \cite{wright-thesis}, because either of these 
hypotheses suffices to imply  that the 
universal cover of the  manifold is hyper-K\"ahler. 

However, while the mathematical papers  cited above 
generally  streamline their terminology by building  the  assumption of being hyper-K\"ahler  into
the very definition of a  ``gravitational instanton,''  Gibbons, Hawking,  and their physicist colleagues certainly never  intended  for
the term to be narrowed in this way.  Indeed, one of Hawking's first examples \cite{hawking} 
of a gravitational instanton was the Riemannian Schwarzschild metric, 
a complete Ricci-flat manifold diffeomorphic to $S^2\times \RR^2$ that is gotten from the simplest Lorentzian  black-hole solution 
by  formally replacing time $t$ with $it$.  This example is not even locally hyper-K\"ahler, but, in view with its close relationship
with gravitational physics, it most certainly deserves to be called  a ``gravitational instanton.''  The same must also be said
of the more general Riemannian Kerr metrics, which are again complete Ricci-flat metrics on $S^2\times \RR^2$, and which are
formally obtained from 
their Lorentzian spinning-black--hole analogues by multiplying both time and the angular-momentum   parameter 
 by $i$. 

While these last-mentioned  metrics  are not even locally  K\"ahler, they turn out to be {\em conformally} K\"ahler, and so are, in particular,
 {\em Hermitian}.  They are also asymptotically locally flat (ALF), in the  sense of Definition \ref{alfie} below; 
 in particular,  they have cubic volume growth. 
 Finally, they are {\em toric}, in the sense that their isometry groups contain a $2$-torus $\mathbb{T}^2$. 

Complete oriented Ricci-flat Riemannian manifolds with all of these properties were recently classified \cite{bg-toric} by the first two authors of this
paper. They fall into just four smooth connected families, thereby  realizing exactly four different diffeotypes. In addition to the Kerr family
alluded to above, the other possibilities are the Taub-bolt metric, the reverse-oriented Taub-NUT metric, and a family discovered by 
Chen and Teo \cite{chen-teo} in 2011. This Chen-Teo family had been completed unanticipated by the physics community, 
and the  Hermitian nature  of these metrics was only discovered later, by Aksteiner and Andersson \cite{akan}.
In fact, Biquard and Gauduchon  \cite{bg-toric} did  much more than merely classify such metrics; indeed, 
by building on earlier work by Paul Tod \cite{tod-ricci},
they  showed that these metrics can always be constructed   from axisymmetric  harmonic functions on Euclidean $3$-space.

However, 
one disquieting  feature of this otherwise compelling  story is that it  only  concerns  metrics
that are  invariant under an isometric $\mathbb{T}^2$-action. The aim of this paper 
is to replace this  symmetry assumption  with an open, purely Riemannian condition. 
Indeed, the  Ricci-flat metrics occurring in the above classification turn out to all
have the property that their self-dual Weyl tensors $W^+: \Lambda^+\to \Lambda^+$
satisfy $\det (W^+) > 0$ everywhere. Einstein metrics with this property are said
to satisfy {\em Wu's criterion}, in honor of Peng Wu \cite{pengwu}, who first discovered 
that compact Einstein manifolds with this property are necessarily conformally K\"ahler. 
Our main observation here is that the first author's proof \cite{lebdet} of Wu's criterion can be adapted to  the context of ALF gravitational 
instantons, provided one imposes  fall-off conditions on the metric that  are stringent enough to provide good   control the boundary terms. 
Our first main result is the following:

     \setcounter{main}{0}       
   \begin{main}
    Let $(M,h_0)$ be any toric, Hermitian, but non-K\"ahler  ALF gravitational instanton. If $h$ is another 
     Ricci-flat Riemannian metric on $M$ that  is sufficiently $C^3_1$-close to $h_0$, 
then       $(M,h)$ is also  a   Hermitian   ALF gravitational 
instanton, and   carries a non-trivial  Killing field $\xi$. 
Moreover, $h$ is  conformally related  to a complete, strictly extremal K\"ahler metric $g$.
 \end{main}

For the proof,  see \S \ref{bullseye} below, where the $C^3_1$ norm  is also defined. 

\medskip 

For two  of the four families of toric Hermitian gravitational instantons, the metric satisfies  both $\det(W^+)> 0$ and $\det (W^-)> 0$,
and  our methods can therefore be applied for both orientations. Doing so then  leads to a stronger result in these cases: 

    \setcounter{main}{1}       
  \begin{main}
 Let $(M,h_0)$ be 
  a Kerr or Taub-bolt  gravitational instanton, 
  and let $h$ be another Ricci-flat metric on $M$ that is sufficiently $C^3_1$-close to $h_0$. 
  Then $(M,h)$ is once again a Kerr or Taub-bolt gravitational instanton. 
    \end{main}

For the proof, see  \S \ref{wagyu} below. That section also highlights   the result of   
Aksteiner, Anderson, Dahl, Nilsson and Simon \cite{AADNS} which  originally  led us to expect for Theorem \ref{bambino} 
to hold. In addition,  we have  added a brief discussion of a recent preprint\footnote{The first version of our own paper was submitted to the {\sf arXiv} during the  weekend  between the submission of   Li's e-print  and its public posting.}  by Minyang Li \cite{mingyang}  that    indicates that Theorem \ref{bambino} could be generalized to also cover   the   two other
 families of gravitational instantons that  appear  in   the Biquard-Gauduchon classification. 

\section{Wu's Criterion, Revisited} 
\label{woohoo} 

In  this section, we will see  that 
 a  remarkable open criterion introduced  by  Peng Wu \cite{pengwu} in the context 
compact Einstein $4$-manifolds also leads to compelling results in other contexts. 
Wu originally  observed  that a compact oriented  simply-connected Einstein $4$-manifold with   Einstein constant $\lambda > 0$ has 
$\det (W^+) > 0$ if the metric is conformally K\"ahler,  
and  then gave a rather opaque argument to show that the converse is also true. 
This prompted the  third author of the present paper to give an entirely  different proof   \cite{lebdet}  of Wu's 
criterion that actually  proves more;  for example, it turns out  that 
a compact oriented 
Riemannian $4$-manifold $(M,h)$ with $b_+(M) \neq 0$ 
satisfies   $\delta W^+=0$ and $\det (W^+) > 0$  if and only if $b_+(M) = 1$ and $h=s^{-2}g$ for some  K\"ahler metric 
$g$ on $M$ with scalar curvature $s > 0$. 
In what follows, we 
will  localize many key  steps  in \cite{lebdet} in order to 
obtain results that are better adapted    to the study  of gravitational instantons.

Suppose that $(M,h)$  is an oriented Riemannian $4$-manifold whose {self-dual Weyl curvature tensor} $W^+$ is 
 {\em harmonic}, in the sense that 
 \begin{equation}
\label{hsdw}
\delta W^+:= -\nabla \cdot W^+=0.
\end{equation}
For example, if $h$ is Einstein, equation \eqref{hsdw} follows as  a consequence of the second Bianchi identity.
However, independent of 
such considerations, the Dirac-type equation  \eqref{hsdw} always implies   \cite[equation (6.8.40)]{pr2}  that
\begin{equation}
\label{roger} 
0 = \nabla^*\nabla W^++ \frac{s}{2} W^+ - 6 W^+\circ W^+ +  2|W^+|^2 I ,
\end{equation}
a   Weitzenb\"ock formula that  has often been eclipsed by   its very useful 
  contraction \cite[equation (16.73)]{bes}
with $W^+$. Next, let $f: M \to \RR^+$ be a smooth positive function on $M$, and consider the corresponding conformal rescaling $g=f^{-2}h$ 
of the original metric $h$. Owing to the weighted conformal invariance   \cite[equation (6.8.8)]{pr2} of the 
Dirac-type equation \eqref{hsdw}, one then finds that 
\begin{equation}
\label{rescaled} 
\delta (fW^+)=0
\end{equation}
with  respect to the conformally rescaled metric $g$. 
The  same calculation \cite[equation (6.8.35)]{pr2}  that proves \eqref{roger} therefore now   gives us 
a   Weitzenb\"ock formula  
\begin{equation}
\label{initio}
0 = \nabla^*\nabla (fW^+)+ \frac{s}{2} fW^+ - 6 fW^+\circ W^+ + 2 f|W^+|^2 I 
\end{equation}
for $fW^+$ with respect to the rescaled metric  $g$. 

 We will next need to clearly  understand  when  an oriented Riemannian  $4$-manifold  has 
  $\det (W^+) > 0$. 
   Since $W^+: \Lambda^+\to \Lambda^+$ is  self-adjoint, we can diagonalize  $W^+$ at any point $p\in M$ as 
$$W^+ = \left[\begin{array}{ccc}\alpha &  &  \\ &  \beta&  \\ &  & \gamma\end{array}\right] $$
by choosing a suitable orthonormal basis  for $\Lambda^+$; and, after re-ordering our basis if necessary, we may arrange that 
$\alpha \geq \beta \geq \gamma$ at $p$. 
However,  by its very definition, 
   the self-dual Weyl curvature $W^+: \Lambda^+\to \Lambda^+$ 
  automatically satisfies  $\trace (W^+) =0$, and it therefore follows that 
$$\alpha  +  \beta + \gamma =0.$$
Hence  $\alpha > 0$ and  $\gamma < 0$ at any point where   $W^+\neq 0$, and we thus see 
 that   $\det W^+ =  \alpha \beta  \gamma$ always  has the same sign as 
$-\beta$, where $\beta$ is once again the middle eigenvalue. 
In other words,  $\det W^+ > 0$ everywhere iff 
 {\sf exactly one} of the eigenvalues, namely $\alpha$, is positive 
at each point, while both the other two  are  negative:
$$W^+ \sim \left[\begin{array}{ccc} + &  &  \\ &  -&  \\ &  & -\end{array}\right] .$$
 This in particular implies that  the positive eigenvalue $\alpha$ 
has multiplicity one everywhere. If we let  $S(\Lambda^+) \subset \Lambda^+$ 
denote the sphere bundle defined by  $|\omega |^2 =2$, then the 
the smooth function 
$\mathsf{Q}  : S(\Lambda^+) \to \RR$ defined by $\mathsf{Q}  (\omega ) = W^+(\omega , \omega )$ 
 has non-degenerate fiberwise Hessian along the set $\mathscr{P}\subset S(\Lambda^+)$ 
of fiberwise maxima. Equivalently, the restriction of $d\mathsf{Q} $ to the 
fibers, considered as a section of the vertical cotangent bundle of $S(\Lambda^+)$, is transverse along $\mathscr{P}$ 
to the zero section of this vector bundle.
The implicit function theorem therefore guarantees that $\mathscr{P}$  is a  submanifold of $S(\Lambda^+)$, and  that it is moreover  transverse to the fibers of $S(\Lambda^+)\to M$.
It therefore follows  that the set $\mathscr{P}$ of $\alpha$-eigenforms $\omega\in S(\Lambda^+)$ can locally be parameterized by a system of 
smooth local sections of $\Lambda^+\to M$.  Since there are exactly two choices of such an $\omega$ at each point of $M$, 
differing only by sign, we conclude that $\mathscr{P}:= \{ \omega \in \Lambda^+~|~ W_h^+(\omega ) = \alpha_h \omega , \quad |\omega |_h^2 =2\}$
is   a smooth 
principal
 $\ZZ_2$-bundle over $M$. More importantly, 
 $\alpha : M\to \RR$  actually defines a smooth positive function $\alpha = W^+(\omega , \omega ) /2$ on $M$,  obtained by 
 taking $\pm \omega$ to be the two  local sections of $\mathscr{P}\subset \Lambda^+$ near an arbitrary point of $M$.

 \begin{rmk} Because $x=\alpha$ is the unique positive solution of the depressed cubic equation 
\begin{equation}
\label{cubic}
 0=  \det (xI-W^+ ) = x^3 - \left[ \frac{1}{2} |W^+|^2\right] x - \det (W^+) , 
\end{equation}
Cardano's formula  provides an explicit expression 
\begin{equation}
\label{tart}
\alpha = 2^{2/3} \Re e \, \sqrt[3]{ \det (W^+) + i \sqrt{ \frac{|W^+|^6}{54}- \left[\det (W^+)\right]^2 }}
\end{equation}
   for the eigenvalue $\alpha$ in terms of $|W^+|^2$ and $\det(W^+)$, 
   where the cube root  means  the principal branch, with positive real part.  However, this explicit formula does not make it  
   obvious  that $\alpha$ remains  smooth  at points where $|W^+|^3 = 3\sqrt{6} \det(W^+)  > 0$, or in other words,  where 
   $\beta = \gamma < 0$. Because the smoothness of  $\alpha: M \to \RR^+$ plays such a central role in what follows, 
   we have therefore chosen to emphasize the above   abstract  explanation of  its  regularity, rather than focusing on the explicit formula 
    \eqref{tart}.
 \end{rmk}
 
 For simplicity, we  will henceforth assume that $M$ is simply connected. This then 
  guarantees that  the principal
 $\ZZ_2$-bundle $\mathscr{P}\to M$ is in fact trivial. Consequently, there is  then a smooth globally-defined  self-dual $2$-form $\omega$ 
 on $M$ with $W^+(\omega ) = \alpha \omega$ and $|\omega|^2 = 2$ at every point; moreover, the only other such 
global  $2$-form is $-\omega$, so this $\omega$ is actually 
 unique up to overall sign. Equivalently, 
 our assumption  that $M$ is simply connected implies that 
  the $\alpha$-eigenspace $L\subset \Lambda^+$ of $W^+$ 
is   a trivial real line-bundle  $L\to M$.

Now  the  condition $\det (W^+) > 0$ is  conformally invariant,  so  the above discussion also  applies to the 
conformal rescaling  $g= f^{-2} h$ of  $h$ defined by any  smooth positive function  $f: M\to \RR^+$. 
  However,  the endomorphism $W^+: \Lambda^+ \to \Lambda^+$ is defined by raising an index
  $$\varphi_{ab}\longmapsto [W^+(\varphi)]_{cd} := \frac{1}{2} {W^{+ab}}_{cd}~\varphi_{ab},$$
  and so carries a conformal weight, even though the  na\"{\i}ve Weyl tensor ${W^{+a}}_{bcd}$ is literally conformally invariant.
  Consequently, replacing $h$ with $g=f^{-2}h$ rescales the top eigenvalue 
  by a factor of $f^2$:
  $$\alpha_g = f^{2}\alpha_h.$$
  We will henceforth impose the interesting  choice  
  \begin{equation}
\label{shrewd}
\boxed{f= \alpha_h^{-1/3}}
\end{equation}
  for our  conformal factor $f$, which then   has the 
 consequence  that 
  $$ \alpha_g = f^2\alpha_h= \alpha_h^{1/3}= f^{-1}.$$
In particular,  $\alpha := \alpha_g$  therefore  satisfies
  \begin{equation}
\label{cunning}
\alpha f \equiv 1
\end{equation}
for this carefully chosen  conformally-rescaled metric $g=f^{-2}h$. 
   
  Since our assumption that $M$ is simply-connected again  implies that  the $\alpha$-eigenbundle  $L\subset \Lambda^+$ of $W^+$ is a  trivial real line-bundle, 
 there once again  exists a global  self-dual $2$-form $\omega$ on $M$ that satisfies
\begin{equation}
\label{charlie}
W^+_{g}(\omega) = \alpha_g ~\omega , \qquad |\omega|_g^2 = 2
\end{equation}
at every point of $M$; moreover, this $\omega$ is unique up to overall sign. 
 Here, normalizing  $\omega\in \Lambda^+$ so that  $|\omega|_g^2=2$ is equivalent to requiring   that 
$$\omega= g(J\cdot , \cdot )$$
for a unique  almost-complex structure $J$ that is compatible with both  $g$ and the given orientation. 
If we can show, under suitable circumstances, that 
$\nabla \omega =0$ with respect to the Levi-Civita connection $\nabla$ of $g$, it will then follow that 
 $J$ is integrable,  and that $(M,g,J)$ is K\"ahler, with  K\"ahler form $\omega$. 
 When  $\delta W^+=0$, our strategy for proving this will be based on a careful  study of the inner product 
\begin{equation}
\label{proximo} 
0= \Big\langle \nabla^*\nabla (fW^+) + \frac{s}{2} fW^+ - 6 fW^+\circ W^+ + 2 f|W^+|^2 I ,~ \omega \otimes
\omega \Big\rangle
\end{equation}
of \eqref{initio} with $\omega \otimes \omega$. 
To make headway on this, we will need  a few key facts about self-dual $2$-forms and the Weyl curvature, starting with  the following:
\begin{lem} 
\label{bliss}
Let $(M,h)$ be a simply-connected oriented  Riemannian $4$-manifold for which  $\det (W^+)> 0$ everywhere.   
 Let $g=f^{-2}h$ be a conformal rescaling of $h$,   and  let $\omega$   be a self-dual $2$-form   that  satisfies \eqref{charlie} everywhere. 
Then 
\begin{equation}
\label{negatron} 
  W^+ ( \nabla^a \omega , \nabla_a\omega) \leq 0 ,
\end{equation}
everywhere, where every term is understood to be defined with respect to $g$. 
\end{lem}
\begin{proof} The covariant derivative   $\nabla \omega$ of $\omega$ belongs to $\Lambda^1 \otimes \omega^\perp \subset\Lambda^1 \otimes  \Lambda^+$ because
$\omega$ has constant norm with respect to $g$. The result therefore follows from the fact that $W^+(\phi , \phi) \leq 0$ for any 
$\phi \in \omega^\perp \subset \Lambda^+$. 
\end{proof} 

\noindent 
We will also need  the following standard algebraic observation:
\begin{lem} At any point $p$ of  an oriented Riemannian $4$-manifold $(M,g)$, 
\begin{equation}
\label{megatron} 
|W^+|^2 \geq \frac{3}{2} \alpha^2
\end{equation}
where $\alpha = \alpha_g$ is the  the top eigenvalue of $W^+_{g}$ at $p$.
\end{lem}
\begin{proof}
Because $\trace \, (W^+)=0$, 
$$
|W^+|^2 = \alpha^2 + \beta^2 + (-\alpha - \beta)^2 = \frac{3}{2}\alpha^2 +2 (\beta + \frac{1}{2} \alpha)^2  \geq \frac{3}{2}\alpha^2
$$
where $\beta$ is  the middle  eigenvalue of $W^+_g$ at $p$.
\end{proof} 
\noindent  Finally, 
 we will also need  the Weitzenb\"ock formula \cite[equation (6.8.38)]{pr2} 
 \begin{equation}
\label{harmless} 
(d+d^*)^2 \omega = \nabla^*\nabla \omega - 2W^+(\omega) + \frac{s}{3}\omega
\end{equation}
for the Hodge Laplacian on   self-dual $2$-forms; cf. \cite[p. 324]{besse4}.

\begin{lem} 
\label{grapple}
If $(M,h)$ is a simply-connected  oriented Riemannian $4$-manifold that  satisfies 
$\det (W^+) > 0$, 
then the 
conformally  rescaled metric  $g=f^{-2}h$  defined by \eqref{shrewd} and  the self-dual $2$-form $\omega$ defined by \eqref{charlie} together satisfy 
\begin{equation}
\label{surprise}
\langle \nabla^*\nabla (fW^+), \omega \otimes \omega \rangle  
\geq  2 |\nabla \omega |^2   , 
\end{equation} at every point of $M$, where both sides are computed relative to $g$.
\end{lem}
\begin{proof}
Since $f\alpha :=f\alpha_g \equiv 1$, $W^+(\omega ) = \alpha \omega$,  and $|\omega|^2 \equiv 2$, we have 
\begin{eqnarray*}
\langle \nabla^*\nabla (fW^+), \omega \otimes \omega \rangle  &=&  \langle -\nabla^a \nabla_a (fW^+), \omega \otimes \omega \rangle  
\\
&=& -\nabla^a \nabla_a \langle  fW^+, \omega \otimes \omega \rangle   + 2 \nabla_a \langle fW^+, \nabla^a(\omega \otimes \omega )\rangle 
\\&&
 \qquad  \qquad\qquad -  f \langle W^+, \nabla^a\nabla_a (\omega \otimes \omega )\rangle 
 \\
&=& \Delta ( f \alpha |\omega|^2 )    + 4 \nabla_a \langle fW^+, \omega \otimes \nabla^a \omega \rangle 
\\&&
 \qquad  \qquad\qquad  - f \langle W^+, \nabla^a\nabla_a (\omega \otimes \omega )\rangle 
  \\
&=& 4 \nabla_a \langle  ( f\alpha) \omega ,  \nabla^a \omega \rangle 
  - f \langle W^+, \nabla^a\nabla_a (\omega \otimes \omega )\rangle 
 \\
&=&  2 \nabla_a \nabla^a |\omega|^2
  - f \langle W^+, \nabla^a\nabla_a (\omega \otimes \omega )\rangle 
   \\
&=&  
  - f \langle W^+, \nabla^a\nabla_a (\omega \otimes \omega )\rangle 
\\&=&  
    - 2f \langle W^+, \omega \otimes \nabla^a\nabla_a \omega \rangle  -2 f \langle W^+, (\nabla^a\omega) \otimes (\nabla_a \omega )\rangle
    \\&=&  
    - 2(f\alpha) \langle \omega , \nabla^a\nabla_a \omega \rangle  -2 f  W^+ (\nabla^a\omega , \nabla_a \omega )
       \\&=&  
   \Delta |\omega|^2 +   2|\nabla \omega|^2  -2 f  W^+ (\nabla^a\omega , \nabla_a \omega )
     \\&=&  
      2|\nabla \omega|^2  -2 f  W^+ (\nabla^a\omega , \nabla_a \omega )
\end{eqnarray*}
Since  $\det (W^+) > 0$,  Lemma \ref{bliss}  then implies  the promised inequality \eqref{surprise}. 
\end{proof}

\begin{prop} 
\label{handhold}
Let $(M,h)$ be a simply-connected  oriented Riemannian $4$-manifold that satisfies $\delta W^+=0$ 
and  $\det (W^+) > 0$ everywhere. Then the 
conformally-rescaled metric  $g=f^{-2}h$  defined by \eqref{shrewd} and  the self-dual $2$-form $\omega$ defined by \eqref{charlie} together satisfy 
\begin{equation}
\label{gobsmacked}
0\geq  \frac{1}{2}  |\nabla\omega |^2 ~   +
 \frac{3}{2}   
\Big\langle \omega , (d+d^*)^2 \omega \Big\rangle 
\end{equation} at every point of $M$, where all terms are to be computed with respect to $g$.
\end{prop}

\begin{proof} By  applying 
\eqref{cunning}, \eqref{charlie},  \eqref{negatron},  \eqref{megatron},  \eqref{harmless},  and \eqref{surprise} to   \eqref{proximo}, we have 
\begin{eqnarray*} 0&=& 
 \Big\langle \nabla^*\nabla fW^+ + \frac{s}{2} fW^+ - 6 fW^+\circ W^+ + 2 f|W^+|^2 I , \omega \otimes
\omega \Big\rangle ~ \\
&=& \langle \nabla^*\nabla ( fW^+) , \omega\otimes \omega \rangle  +  \Big[ \frac{s}{2} W^+(\omega , \omega ) 
 - 6 |W^+(\omega)|^2+ 2 |W^+|^2 |\omega |^2 \Big] f~   \\
&\geq & 2|\nabla \omega |^2  
  +  \Big[ \frac{s}{2} \alpha |\omega|^2 
 - 6 \alpha^2 |\omega|^2+ 2 |W^+|^2 |\omega |^2 \Big] f~   \\
 &\geq &   2|\nabla \omega |^2  
+  \Big[ \frac{s}{2} \alpha |\omega|^2 
 - 6 \alpha^2 |\omega|^2+ 3 \alpha^2 |\omega |^2 \Big] f~   \\
  &\geq &   2|\nabla \omega |^2  
+  \Big[ \frac{s}{2}  |\omega|^2 
 - 3 \alpha |\omega|^2 \Big] (\alpha  f)~   \\
  &= &   2|\nabla \omega |^2  
+  \Big[ \frac{s}{2}  |\omega|^2 
 - 3 \alpha |\omega|^2 \Big]   \\
  &=&  
\frac{1}{2} |\nabla \omega |^2    + \frac{3}{2} \langle \omega , \nabla^*\nabla \omega  \rangle   -\frac{3}{4}\Delta |\omega|^2
+  \Big[ \frac{s}{2}  |\omega|^2 
 - 3 \alpha |\omega|^2 \Big]   \\
   &=&  
\frac{1}{2} |\nabla \omega |^2 + \frac{3}{2}  \Big[ \langle \omega , \nabla^*\nabla \omega  \rangle -2 W^+(\omega , \omega ) +  \frac{s}{3}  |\omega|^2 
 \Big] \\
  &=& 
 \frac{1}{2}  |\nabla\omega |^2 ~   +
 \frac{3}{2}   
\Big\langle \omega , (d+d^*)^2 \omega \Big\rangle ,
 \end{eqnarray*}
 thus proving the promised pointwise inequality. 
    \end{proof} 

\begin{lem} 
\label{lifeline} 
Under the hypotheses of Proposition \ref{handhold}, the $2$-form $\omega$  satisfies 
$$
(2\sqrt{6}|W^+| - s ) \geq 2 |\nabla \omega |^2 
$$
with respect to the rescaled metic $g$, at every point of  $M$. 
\end{lem}
\begin{proof} Since $|\omega|^2 \equiv 2$, the proof of Proposition \ref{handhold} shows, in particular, that 
$$0\geq 2|\nabla\omega |^2 +\left( \frac{s}{2} - 3\alpha \right)|\omega|^2 = 2|\nabla\omega |^2 +(s- 6\alpha )$$
everywhere. Because $\sqrt{\frac{2}{3}} |W^+| \geq \alpha$ by 
\eqref{megatron}, 
the claim therefore follows. 
\end{proof}

On the other hand, notice that 
\begin{eqnarray*}
\star \Big\langle \omega , (d+d^*)^2 \omega \Big\rangle&=& \omega \wedge  (d+d^*)^2 \omega \\
&=& \omega \wedge [d(-\star d\star )+ (-\star d \star) d] \omega \\
&=& - 2 \omega \wedge [d  \star d  \omega ]\\
&=&\star  2 |d\omega|^2  - 2 ~d [ \omega \wedge \star d  \omega ].
\end{eqnarray*}
This allows us to rewrite  the pointwise inequality \eqref{gobsmacked}
 as 
\begin{equation}
\label{implement} 
3 \, d [ \omega \wedge \star d  \omega ]\geq \star \Big(\frac{1}{2} |\nabla \omega |^2  +3\, |d\omega|^2\Big).
\end{equation}
Integrating \eqref{implement}  on a precompact domain with smooth boundary, and then applying Stokes' Theorem, 
we therefore obtain the following result:

\begin{prop}  Let $(M,h)$ be a simply-connected   oriented Riemannnian $4$-manifold that satisfies $\delta W^+=0$ and $\det (W^+) > 0$.
Consider the  conformally  rescaled metric  $g=f^{-2}h$  defined by \eqref{shrewd}, and let $\omega$ be one of the two 
 self-dual $2$-forms  on $M$ that satisfies  \eqref{charlie}.
Then for any precompact domain  $U \subset M$ with smooth boundary $\partial U= \overline{U} -U$, we have 
 \begin{equation}
\label{hologram}
3\int_{\partial U}  \omega \wedge \star d  \omega\geq \int_{U} \Big[  \frac{1}{2}  |\nabla\omega |^2 ~   + 3\, |d\omega|^2 \Big] d\mu_g .
 \end{equation}
\end{prop}

This now  allows us to  deduce the following:

  \begin{prop} 
\label{harness}
Let $(M,h)$ be an  oriented,  simply-connected Riemannian $4$-manifold that satisfies $\delta W^+=0$ 
and  $\det (W^+) > 0$ everywhere. Suppose, moreover, that 
 $M$ 
  is a nested union $M=\cup_j U_j$ of precompact domains  $U_1\Subset U_2 \Subset \cdots \Subset U_j\Subset \cdots$ 
 with smooth boundary
such that 
$$\lim_{j\to\infty} \int_{\partial U_j}  \omega \wedge \star d  \omega = 0,$$
where $\omega$  is   the self-dual $2$-form  defined by \eqref{charlie} relative to 
 the 
conformally  rescaled metric  $g=f^{-2}h$  defined by \eqref{shrewd}.
Then $(M,g)$ is a K\"ahler manifold.
\end{prop}
\begin{proof} 
Applying \eqref{hologram} to each   $U_j$  yields 
$$3\int_{\partial U_j}  \omega \wedge \star d  \omega\geq \int_{U_j} \Big[  \frac{1}{2}  |\nabla\omega |^2 ~   + 3 |d\omega|^2 \Big] d\mu_g \geq 
\frac{1}{2}  \int_{U_j} |\nabla\omega |^2 d\mu_g  . 
$$
Since the right-hand side is non-negative, our hypothesis therefore implies that 
 $\lim_{j\to \infty}\int_{U_j} |\nabla \omega|^2 d\mu_g =0.$ But  $U_j \subset U_{j+1}$,  so
  the terms $\int_{U_j} |\nabla \omega|^2 d\mu_g$  in this sequence  are also  non-decreasing in $j$. Thus 
 $\int_{U_j} |\nabla \omega|^2 d\mu_g =0$ for all $j$, and hence $\nabla \omega \equiv$ on each $U_j$. But since 
  $M= \cup_j U_j$, this implies  that $\nabla \omega \equiv 0$ on all of $M$. It follows  that  $(M, g)$ is a K\"ahler manifold. 
   \end{proof}

 \section{Gravitational Instantons} \label{bullseye}
  
In order to  be able to invoke  Proposition \ref{harness} in concrete circumstances, we next show that the relevant boundary hypothesis
will automatically hold if certain geometric  conditions are fulfilled.

 \begin{thm}
 \label{punch}
  Let $(M,h)$ be an oriented, simply-connected, Ricci-flat\linebreak  $4$-manifold that   satisfies  $\det (W^+) > 0$ everywhere,
  and suppose that
  $M$ is expressed 
  as a nested union $M=\cup_j U_j$ of precompact domains $$U_1\Subset U_2 \Subset \cdots \Subset U_j\Subset \cdots$$ with smooth boundary.
   Let 
  $g=f^{-2}h$ be the conformally rescaled metric defined by \eqref{shrewd}, and  let  $d\check{\mu}_g$ denote the $3$-dimensional volume
  measure on each of   these boundaries $\partial U_j$ induced by the restriction of $g$. Also suppose that the $g$-induced 
   $3$-dimensional volumes of these boundaries 
    are uniformly bounded, 
 while  the integrals of $|W^+_g|$ and $s_g$ on these boundaries tend to zero with respect to this same $3$-dimensional volume-measure: 
 \begin{eqnarray}
  \int_{\partial U_j} 1~ d\check{\mu}_g  &<& \mathsf{C}, \label{wynken}\\
 \lim_{j\to \infty} \int_{\partial U_j} |W^+_g| ~d\check{\mu}_g &=& 0,\label{blynken}\\
  \lim_{j\to \infty} \int_{\partial U_j} |s_g| ~d\check{\mu}_g &=& 0.\label{nod}
\end{eqnarray}
  Then $(M,g)$ is a  strictly  extremal K\"ahler manifold, while   the given Ricci-flat $4$-manifold 
  $(M,h)$ is Hermitian, and carries a non-trivial Killing field. 
   \end{thm}  
   \begin{proof} Because $\omega$ is a self-dual $2$-form of norm $\equiv \sqrt{2}$, one has 
     $$
    |\omega \wedge \star d\omega|= |\delta \omega| =|\nabla\cdot \omega |
   $$
   with respect to $g$,    and it is therefore relatively easy to  see   that
   \begin{equation}
   \label{loose} 
2\sqrt{2} |\nabla \omega| \geq |\omega \wedge \star d\omega|
\end{equation}
 at  every point of $M$. (Indeed,  after    detailed calculation,  the constant $2\sqrt{2}$  in \eqref{loose} can actually be replaced  by  $1$; 
     but  the gist of what  follows merely  depends on the fact that there 
  is {\sf some} universal constant for which such an   inequality  holds.) Now, 
    Lemma \ref{lifeline} tells us that the inequality 
      $$(2\sqrt{6} |W^+| - s) \geq 2|\nabla\omega|^2$$
      also holds at every point of $M$. Consequently, inequality   \eqref{loose} implies that
      $$2 \left[2\sqrt{6} |W^+| + |s|\right]^{1/2} \geq |\omega \wedge \star d\omega |$$
     at every point of $M$, and  we therefore deduce that 
     $$ 2\int_{\partial U_j} \left[2\sqrt{6} |W^+| + |s|\right]^{1/2} d\check{\mu}_g  \geq \left|  \int_{\partial U_j}  \omega \wedge \star d\omega  \right|
     $$ 
     holds for each $j$. 
    The  Cauchy-Schwarz inequality thus implies   that 
        $$2\left[ \int_{\partial U_j} 1~ d\check{\mu}_g \right]^{1/2} \left[2\sqrt{6}  \int_{\partial U_j} |W^+| ~ d\check{\mu}_g + \int_{\partial U_j}|s| ~d\check{\mu}_g \right]^{1/2} \geq \left|  \int_{\partial U_j}  \omega \wedge \star d\omega  \right|$$
     for every $j$.  
       Thus, our hypotheses \eqref{wynken}, \eqref{blynken}, and \eqref{nod}   now imply that 
     $$\lim_{j\to\infty} \int_{\partial U_j}  \omega \wedge \star d  \omega = 0.$$

     Since 
$(M,h)$ also satisfies $\delta W^+=0$ and  $\det(W^+) > 0$, this means that all the hypotheses of    Proposition \ref{harness} 
are  fulfilled, and it therefore follows that $g$ is  a K\"ahler metric. 
   In particular, the conformally related metric $h=f^2 g$ is 
   Hermitian. However, since  $h$ is also Ricci-flat, and hence Bach-flat,  the 
   conformal invariance of the latter condition guarantees that $g$ is Bach-flat, too. This  forces \cite{derd,leb-bfk} 
   the K\"ahler metric $g$ to be extremal in the sense of Calabi. However, 
   since $\det (W^+) > 0$, the K\"ahler metric $g$ also satisfies $s/6= \alpha = f^{-1} > 0$.
   But $h= f^2g$ is scalar-flat, whereas $g$ has positive scalar curvature. 
   The conformal factor $f=6s^{-1}$, and hence the scalar curvature $s$ of $g$,  must therefore be non-constant.
  This shows that    $g$ must 
    be a {\em strictly}  extremal K\"ahler metric. In particular, $\xi :=J\nabla s$ must be a non-trivial Killing field --- not only for $g$, but also  for its $\xi$-invariant conformal rescaling 
   $h= 36s^{-2}g$.
   \end{proof}
   
Our goal is now to apply   Theorem \ref{punch}  to   asymptotically locally flat (ALF) gravitational instantons, thereby    putting   the 
 first two authors'   main  classification result  \cite[Theorem 8.2]{bg-toric}    in a new and broader context.

   \begin{defn} 
   \label{alfie} 
   Let $(M,h)$ be a complete, Ricci-flat  Riemannian $4$-manifold $(M,h)$. Then 
  $(M,h)$ will be called an  {\sf ALF gravitational instanton} if  
  \begin{itemize}
  \item   there is a compact subset $\mathfrak{C} \subset M$
   such that $M-\mathfrak{C}$ is diffeomorphic to $\RR^+ \times \Sigma$, where 
    $\Sigma^3$ is  oriented, and  finitely covered by $S^2\times S^1$ or $S^3$; 
    \item $\Sigma$ is equipped with a sign-ambiguous pair $\pm (T,\eta)$ of a vector field $T$ and a $1$-form $\eta$ 
    which  satisfy    $T\lrcorner\, \eta =1$ and
    $T\lrcorner \, d\eta=0$ on, at worst, a double cover of $\Sigma$; 
  \item $\Sigma$ is also equipped with a positive-semi-definite symmetric $2$-tensor field $\gamma\in \Gamma (\odot^2 T^*\Sigma )$
  such that  $\mathcal{L}_T\gamma=0$  and $\ker \gamma= \espan T$, and which moreover locally defines a 
   Gauss-curvature $+1$ metric on the space of leaves of  the foliation tangent to $T$;   and 
  \item  after pulling back via  a suitable   diffeomorphism $\RR^+ \times \Sigma \to M-\mathfrak{C}$, the metric $h$  takes the form
$$ h = d\varrho^2 + \varrho^2 \gamma + \eta^2 + \mho, $$
  where 
the standard coordinate $\varrho$  on $\RR^+$ and the tensor fields  $\gamma$ and $\eta^2$ on $\Sigma$ have been pulled back to $\RR^+ \times \Sigma$
via the first- and second-factor projections, and where the error term $\mho$ satisfies the fall-off condition
\begin{equation}
\label{falloff} 
 \mathbb{D}^j \mho = O(\varrho^{-1-j}) 
\end{equation}
for $0\leq j \leq 3$, where $\mathbb{D}$  denotes 
  the Levi-Civita connection
  of the  background metric 
$d\varrho^2 + \varrho^2 \gamma + \eta^2$ on $\RR^+ \times \Sigma$. 
  \end{itemize}
   \end{defn}
   
   \noindent 
 For the gravitational instantons that will  primarily concern  us here, 
 $T$ and $\eta$ are both single-valued on $\Sigma$, without any need to pass to a double cover. 
  However,  Definition \ref{alfie} has  
 been carefully  worded   to  avoid excluding e.g.\    the ALF hyper-K\"ahler gravitational instantons  \cite{chenchen,minerbeD}   of
 type $D_k$. On the other hand, because Definition \ref{alfie}  assumes  from the outset that $(M^4,h)$ is complete and Ricci-flat, 
we have simplified    \cite[Definition 1.1]{bg-toric} 
  by assuming that $M$  only has one end, since     the Cheeger-Gromoll splitting theorem \cite{cg}
 would  otherwise lead to a  contradiction, e.g. by forcing $(M^4,h)$ to be a flat Riemannian product $\RR\times \Sigma$, 
 even though  $\Sigma$  cannot admit flat metrics.

Now, in the  spirit of our fall-off hypothesis \eqref{falloff},   and after choosing  some base-point $p\in M$, 
 we will say   that a $C^k$ function or tensor field $\mho$ on an ALF gravitational instanton 
 $(M,h_0)$ is 
 of weighted class $C^k_1$ if 
 $$\|\mho \|_{C^k_1} := \sup_M \, \sum_{j=0}^k(1+\dist)^{j+1}  |\nabla^j \mho  |_{h_0} $$
 is finite, where $\nabla$ denotes the Levi-Civita connection of $h_0$. We will also  say  that a second  $C^k$ metric $h$ on $M$ is
  within $C^k_1$-distance $\varepsilon$ 
 of $h_0$ if, in terms of this weighted norm,  $\| h - h_0\|_{C^k_1} < \varepsilon$. 
 Having fixed these conventions, we now formulate and prove a concrete implementation  of Theorem \ref{punch}: 
 
 \pagebreak 
   
     \setcounter{main}{0}       
   \begin{main}
   \label{judy}
    Let $(M,h_0)$ be any toric, Hermitian, but non-K\"ahler  ALF gravitational instanton. If $h$ is another 
     Ricci-flat Riemannian metric on $M$ that  is sufficiently $C^3_1$-close to $h_0$, 
then       $(M,h)$ is also  a   Hermitian   ALF gravitational 
instanton, and   carries a non-trivial  Killing field $\xi$. 
Moreover, $h$ is  conformally related  to a complete, strictly extremal K\"ahler metric $g$.
 \end{main}
 
   \begin{proof} By the main classification result \cite[Theorem 8.2]{bg-toric}  of Biquard and Gauduchon,
   the given gravitational instanton $h_0$ must be a Kerr, Chen-Teo, Taub-bolt, or (reverse-oriented) Taub-NUT metric. 
   Because these are all ALF gravitational instantons in the sense of Definition \ref{alfie}, their curvature tensors  
  $\mathcal{R}$ satisfy $\mathcal{R} = O (\varrho^{-3})$ and $\nabla \mathcal{R} = O (\varrho^{-4})$.
  Moreover, 
  each of these metrics is conformal to a K\"ahler metric $g_0= u^{-2}h_0$, 
  where the positive function $u$ is  related to the scalar curvature  of $g_0$ by $u^{-1} = k s_{g_0}$
  for some constant $k$. It follows that     $u$  is  {\sf proper}, because $s_{g_0}\to 0$ at infinity \cite[\S 2]{bg-toric}. 
  This then implies that $s_{g_0}>0$ everywhere\footnote{The fact that $g_0=g_{K}$  has positive scalar curvature
  was never made explicit in \cite{bg-toric}, but was, for example,  implicit  in the  conclusion $A>0$  of \cite[Corollary 5.2]{bg-toric}.}.
  Indeed, since $h_0 = u^2 g_0$ is Ricci-flat, and hence scalar-flat, 
  the Yamabe equation  tells us
  that
  $$0=s_{h_0}u^3  =  (6 \Delta_{g_0} + s_{g_0}) u,$$
 where $\Delta = d^* d = -\nabla \cdot \nabla$ is the geometric Laplacian of $g_0$. 
 But since $u > 0$ is proper, it must have some minimum $p\in M$, and at $p$  we  therefore have 
 $$s_{g_0} (p) =  6 u^{-1} (\nabla \cdot \nabla u) |_p = 6 u^{-1} (\trace \Hess u) |_p \geq 0.$$
 Since $s_{g_0}$ is continuous and  nowhere zero, this shows  that  $s_{g_0} > 0$ everywhere, as claimed. 
 If we  now  choose the normalization  $k =1/6$, note that  $u$ then becomes the function
 $f_{g_0}=\alpha_{g_0}^{-1}=\alpha_{h_0}^{-1/3}$ assigned to $h_0$ by \eqref{shrewd}.

It follows that $g_0$ and  $h_0$ both satisfy $\det(W^+) > 0$, and indeed that 
$\det (W^+_{h_0}) = \frac{1}{4} \alpha_{h_0}^3= \frac{1}{3\sqrt{6}} |W^+_{h_0}|^3>0$
everywhere. Meanwhile, the fall-off condition \eqref{falloff} 
  implies that the Riemann tensor of $h_0$ satisfies 
  $$
  \mathcal{R} = O (\varrho^{-3})  \quad \text{and} \quad \nabla \mathcal{R} = O(\varrho^{-4}),
  $$
 so we also, for instance,  have $|W_{h_0}^+| = O(\varrho^{-3})$ in the end region. But in fact,   $|W^+_{h_0}|= \sqrt{\frac{3}{2}}\alpha_{h_0} =  
 \sqrt{\frac{3}{2}}\alpha_{g_0}^3$ actually   has {\em precisely} $\varrho^{-3}$ fall-off, because $\alpha_{g_0}$ is 
  an affine defining function of the   moment polygon's ``edge at infinity,'' 
 and so  is asymptotically  greater \cite[p. 394]{bg-toric} than a constant times $\varrho^{-1}$.

  Now suppose that $h$ is a second Ricci-flat metric on $M$ which is  close to $h_0$ in the $C^3_1$ sense. It  then follows that 
  $h$  also satisfies the fall-off conditions \eqref{falloff} relative 
   to the model metric $d\varrho^2 + \varrho^2 \gamma + \eta^2$. In particular,   $(M,h)$ is  itself an ALF gravitational instanton, and 
  its curvature tensor also satisfies $\mathcal{R} = O (\varrho^{-3})$ and 
   $\nabla \mathcal{R} = O (\varrho^{-4})$. 
 Consequently,  the  self-dual Weyl curvature of  $h$  satisfies  the fall-off conditions
  $$
  |W^+_h|=  O ( \varrho^{-3}), \quad \nabla |W^+_h|=  O ( \varrho^{-4}),
  $$
  and 
  $$
  \det (W^+_h) =O ( \varrho^{-9}), \quad \nabla  \det (W^+_h) =O ( \varrho^{-10}).
  $$
 Since these same fall-off rates apply to $h_0$, and since 
 $3\sqrt{6}\det (W^+) = |W^+_{h_0}|^3$ is bounded above and below by  positive constant multiplies of  $(1+\dist )^{-9}$,  
 we will   automatically have $\det(W^+_h) > 0$ 
 and $2|W^+_{h_0}| > |W^+_{h}|> \frac{1}{2}|W^+_{h_0}|$ if $\| h-h_0\|_{C^3_1} < \varepsilon$ for $\varepsilon$ sufficiently small. 
Consequently,  $\alpha_h > |W^+_h| >  \frac{1}{2}|W^+_{h_0}|$ is then also  larger than a positive constant times $\varrho^{-3}$
when  $\varrho \gg 0$. 
 
 Now, after  writing  \eqref{cubic} as 
   $$\left(\alpha^2 - \frac{1}{2} |W^+|^2\right) \alpha = \det (W^+)$$
and then  putting the first  derivative of this equation  in the form 
$$
\left(3\alpha^2- \frac{1}{2}|W^+|^2\right) \nabla \alpha = \frac{1}{2}\alpha \nabla |W^+|^2 + \nabla \det (W^+) \\
$$
the fact that
 \begin{equation}
\label{squeeze}
\det (W^+) > 0\quad \Longrightarrow \quad     \frac{2}{3} |W^+|^2 \geq \alpha^2 >   \frac{1}{2}   |W^+|^2 
\end{equation}
 now yields   {\em a priori} fall-off rates  for $\alpha_h$: 
 \begin{equation}
\label{tumbler}
 \alpha _h= O ( \varrho^{-3}) \quad \text{and} \quad \nabla \alpha _h= O ( \varrho^{-4}).
 \end{equation}
The function  $\alpha_g:= \alpha_h^{1/3}$   consequently has fall-off 
\begin{equation}
\label{gentler}
 \alpha _g= O ( \varrho^{-1}), \quad |\nabla \alpha _g|_h= O ( \varrho^{-2})  
 \end{equation}
and therefore  belongs to $C^1_1$. Moreover, $\alpha_g= \alpha_h^{1/3}$ is  also 
larger than a positive constant times $\varrho^{-1}$ for all $\varrho \gg 0$. 

Using this, we will now show that the scalar curvature $s_g$ is an $L^1$ function on $(M,g)$, or in other words  that 
$$\int_M |s_g|~d\mu_g < \infty .$$
To see this, let us first observe that 
$$|s_g|= 2s_+ - s,$$
where $s_+ = \max (s_g , 0)$ is the positive part of $s_g$. 
For $\rad \in \RR^+$, now let $M_{\rad}\subset M$   be the compact $4$-manifold-with-boundary  that is the union of 
the subset $\varrho^{-1} ((-\infty, \rad ])$ of the end region and the 
compact complement $\mathfrak{C}$ of end. We then have
$$\int_{M_\rad} |s_g|~d\mu_g\leq  2 \int_{M_\rad} s_{+} \, d\mu_g + \left| \int_{M_\rad} s_g ~d\mu_g \right|.$$
It will thus suffice to show that both $\int_{M_\rad} s_{+} d\mu_g$ and $| \int_{M_\rad} s_g d\mu_g|$ 
remain  uniformly bounded as  $\rad\to \infty$. 

We begin by considering $\int s_gd\mu_g$, which we will analyze  by means of  the Yamabe equation
$$s_g ~d\mu_g = 6\alpha_g (\Delta_h \alpha_g) ~d\mu_h ,$$
incorporating   the facts  that $s_h=0$,   $g= \alpha_g^2 h$, and $d\mu_g = \alpha_g^4 d\mu_h$. 
Thus 
$$\int_{M_\rad} s_g d\mu_g= 
6\int_{M_\rad} |\nabla \alpha_g|^2_h~ d\mu_h
-6 \int_{\partial M_\rad} \alpha_g (\nabla_\nu \alpha_g) ~d\check{\mu}_h  ,
$$
and hence 
$$
\left| \int_{M_\rad} s_g ~d\mu_g \right| \leq 6  \int_{M_\rad} |\nabla \alpha_g|^2_h ~d\mu_h  + 6  \int_{\partial M_\rad}\alpha_g \left| \nabla \alpha_g\right|_h ~d\check{\mu}_h  , 
$$
where $d\check{\mu}_h$ is the volume $3$-form
on $\partial M_\rad$  induced by the restriction of $h$, and 
where $\nu$ 
denotes the outward-pointing unit normal of $\partial M_\rad$ with respect to $h$.
 On the other hand, since $(M,h)$ is an ALF gravitational instanton, in the sense of Definition \ref{alfie},  the metric has the asymptotic  form 
$$h= d\varrho^2 + \varrho^2\gamma + \eta^2 + O (\varrho^{-1}). $$
Thus, in the end region, our fall-off  \eqref{gentler} on  $\alpha_g $ guarantees that 
\begin{eqnarray*}
|\nabla \alpha_g|^2_h ~d\mu_h &\leq &\mathsf{B} \varrho^{-2}| d\rho \wedge \varOmega \wedge \eta | \\
\alpha_g |\nabla \alpha_g|_h ~d\check{\mu}_h&\leq &\mathsf{B}  \varrho^{-1} ~| \varOmega \wedge \eta|
\end{eqnarray*}
for all $\varrho \gg 0$, where $\varOmega$ is the area form of $\gamma$, and $\mathsf{B}$ is a sufficiently large  constant. 
Thus  $|\int_{M_\rad} s_g d\mu_g|$ behaves like $\text{\sf const} + O (\rad^{-1})$ for $\rad \gg 0$, and  is
therefore  uniformly bounded in $\rad$. 

On the other hand, 
since $g= \alpha_g^2 h$,  and since, after possibly increasing  the size of $\mathsf{B}$,
 $\alpha_g <   \frac{1}{2} \mathsf{B} \varrho^{-1}$  in the asymptotic region by \eqref{gentler}, 	
 we also have 
 \begin{equation}
\label{squash}
g <   \mathsf{B}^2 \left[ \gamma + \frac{d\varrho^2  + \eta^2}{\varrho^2} \right]
\end{equation}
for all  $\varrho > \varrho_0$, where the pointwise inequality is to be understood in the sense of quadratic forms. 
In particular, the $4$-dimensional volume measure of $g$ satisfies
$$d\mu_g < \mathsf{B}^4 \varrho^{-2} | d\varrho \wedge \varOmega \wedge \eta|$$
for all $\varrho > \varrho_0$. On the other hand, Lemma \ref{lifeline} guarantees that $$2\sqrt{6} ~|W^+_g|\geq  \max (s_g , 0)= s_+$$
at every point. When combined with \eqref{squeeze}, this then implies  that 
$$4\sqrt{3}~\alpha_g \geq s_+$$
everywhere. Since  $\alpha_g <   \frac{1}{2} \mathsf{B} \varrho^{-1}$  in the asymptotic region,	we therefore have
$$s_+ d\mu_g \leq 2\sqrt{3} ~\mathsf{B}^5 ~\varrho^{-3}  | d\varrho \wedge \varOmega \wedge \eta|$$
for all $\varrho >\varrho_0$. It follows that $\int_{M_\rad} s_+ d\mu_g$ behaves like $\text{\sf const} +O(\rad^{-2})$ for $\rad \gg 0$,
and is therefore  uniformly bounded in $\rad$. Since $|\int_{M_\rad} s_g d\mu_g|$  has also been shown to be uniformly bounded, 
we thus conclude that  $\int_{M_\rad} |s_g | ~d\mu_g$ is uniformly bounded, too. Hence $$\int_{M} |s_g | ~d\mu_g = \sup_\rad \int_{M_\rad} |s_g | ~d\mu_g < \infty , $$
as claimed.

However, since  we also know that  $\alpha_g > 2\mathsf{b} ~\varrho^{-1}$ when $\varrho > \varrho_0$, for some positive constant $\mathsf{b} $, 
we  also have 
$$g  >    \mathsf{b} ^2 \left[ \gamma + \frac{d\varrho^2  + \eta^2}{\varrho^2} 
\right]$$
for all $\varrho > \varrho_0$, 
so the 
 $4$-dimensional volume measure $d\mu_g$ and  
the $3$-dimensional volume measure $d\check{\mu}_g$ of  the  $\varrho= \text{\sf const}$ hypersurfaces 
jointly satisfy 
$$d\mu_g > \mathsf{A} \varrho^{-1} |d\varrho \wedge d\check{\mu}_g| = |dt\wedge d\check{\mu}_g|,$$
 for all $\varrho> \varrho_0$, where $\mathsf{A}=\mathsf{b}^4 /\mathsf{B}^3$, and  where we have 
 set $t:= \mathsf{A} \log \varrho$. 
For each $t$, we now let 
$\Sigma_t \approx \Sigma$ be  the level-set  $\varrho = e^{t/\mathsf{A}}$, 
and then 
define 
$$\digamma (t) = \int_{\Sigma_t} |s_g|~d\check{\mu}_g .$$
Setting $t_0 = \mathsf{A} \log \varrho_0$, we  then have 
$$
\int_{t_0}^\infty \digamma (t ) ~dt < \int_{\varrho \geq \varrho_0}  |s_g| ~d\mu_g < \int_M |s_g|~ d\mu_g < \infty .
$$
 Since $\digamma (t)$ is a  continuous positive function, this means that there must be 
an increasing sequence $t_j\to \infty$ with $\digamma (t_j) \to 0$. 
Setting   $\rad_j =  e^{t_j/\mathsf{A}}$ and defining 
$U_j$ to be the interior of $M_{\rad_j}$
thus defines    an exhaustion of $M=\cup_jU_j$ by nested pre-compact domains
 $$U_1\Subset U_2 \Subset \cdots \Subset U_j\Subset \cdots$$ 
with smooth boundary such that   condition \eqref{nod}
is satisfied:
$$\lim_{j\to \infty} \int_{\partial U_j} |s_g| ~ d\check{\mu}_g = 0.$$
On the other hand, since $\alpha_g< \frac{1}{2} \mathsf{B}\varrho^{-1}$,  
inequalities \eqref{squeeze} and \eqref{squash}  
tell us that 
\begin{eqnarray*}
|d\check{\mu}_g|&<&\mathsf{B}^3 \varrho^{-1}  |\varOmega\wedge\eta | \\
|W^+_g| ~|d\check{\mu}_g|&<& \mathsf{B}^4 \varrho^{-2}  |\varOmega\wedge\eta | 
\end{eqnarray*}
for all $\varrho > \varrho_0$, so there is   a positive constant $\mathsf{C}$ such that \begin{eqnarray*}
\int_{\partial U_j} 1~ d\check{\mu}_g &<&\frac{\mathsf{C}}{\rad_j} \\
\int_{\partial U_j}  |W^+_g| ~d\check{\mu}_g &<& \frac{\mathsf{C}}{\rad_j^2}  
\end{eqnarray*}
for all $j\gg 0$. Since $\lim_{j\to \infty} \rad_j = \infty$, this shows that  conditions \eqref{wynken} and \eqref{blynken} 
are also satisfied  by our exhaustion $U_j$ of $(M,g)$.   Theorem \ref{punch} therefore tells us   that 
  $(M,g)$ is a  strictly  extremal K\"ahler manifold, and that   the ALF gravitational instanton  
  $(M,h)$ is  consequently Hermitian, and carries  a non-trivial Killing field $\xi$, as claimed.  
 \end{proof}

 \section{Rigidity Results} \label{wagyu} 
 
 We will now use Theorem \ref{judy} in conjunction with the  Biquard-Gauduchon classification \cite[Theorem 8.2]{bg-toric}
 to prove  various  rigidity results.  To do this, however,  we will first need to check  that these two machines actually mesh correctly.

 \begin{lem} \label{cleared} 
For $(M,h)$   as in Theorem \ref{judy}, the  limit  at infinity of the Killing field $\xi$ is a non-zero constant multiple of  the vector field  
$T$ on $\Sigma$. In particular, 
the  action on $\Sigma$ induced by $\xi$ at infinity   preserves the triple  $(T, \eta,\gamma)$.
 \end{lem} 
 
\begin{proof} By construction, the Killing field $\xi$  is a given by 
$$ \xi^a = {J_b}^ag^{bc}\nabla_c s_g = 6 {J_b}^ag^{bc}\nabla_c \alpha_g = 6 {J_b}^a f^2h^{bc}\nabla_c f^{-1} = - 6 {J_b}^a h^{bc}\nabla_c f,$$
where $f=\alpha_g^{-1}$.  Our asymptotics for $\alpha$ guarantee that, for $\varrho \gg 0$, $\alpha_g$ is bigger than a positive constant times $\varrho^{-1}$, 
and that $|\nabla \alpha_g|_h$ is less than a constant times $\varrho^{-2}$, so it follows that  $|\nabla f|_h= \alpha_g^{-2} |\nabla \alpha_g|_h$ is uniformly bounded. It therefore follows that $|\xi| =6 |\nabla f|_h$ is uniformly bounded, too. 

On the other hand,   since $\alpha_g\to 0$ at infinity, $f = \alpha_g^{-1}$ is a smooth
proper function on $M$, and therefore achieves its  minimum at some $p\in M$. Since $|df|_p=0$, it therefore follows that $\xi=-6J\nabla f$ has
a zero at $p$. The flow of the Killing field $\xi$ therefore preserves the distance to $p$, so it  follows that $\xi$
is orthogonal to every geodesic  passing through this base-point $p$.

However,  because $\xi$ is a Killing field, it automatically satisfies 
\begin{equation}
\label{jack}
\nabla_a\nabla_b \xi^c= {\mathcal{R}^c}_{bad}\xi^d
\end{equation}
on $(M,h)$, 
primarily as a reflection of  the fact that its restriction to any geodesic is a Jacobi field.  Since $|\xi |_h$ is 
uniformly bounded and $|\mathcal{R}|_h = O(\varrho^{-3})$, it therefore follows that $\nabla\nabla \xi = O(\varrho^{-3})$. 
On the other hand, the asymptotic local model metric $d\varrho^2 + \varrho^2 \gamma + \eta^2$ has a Riemannian 
submersion to Euclidean $\RR^3$ whose fibers are tangent to $T$. By restricting this equation to geodesics and integrating, it
 therefore follows that, along any slice transverse to $T$, 
the projection  $\xi \bmod T$ differs from an affine-linear function   $\RR^3\to \RR^3$  by terms of order $\varrho^{-1}$; and since   $\xi \bmod T$
is uniformly bounded, it therefore follows that $\xi \bmod T$ actually tends to a constant (i.e. parallel) vector field on  Euclidean $\RR^3$. 
But since
$\xi$ is orthogonal to every geodesic through $p$, and since the tangent directions of such geodesics project to an open cone
in $\RR^3$, this leads to an immediate  contradiction unless this constant field is zero. This shows that $\xi = uT + \mho$
for some smooth bounded function $u$, where the error term $\mho$ satisfies 
$\mho = O(\varrho^{-1})$ and $\nabla \mho = O(\varrho^{-2})$. 
However, because $T$ is a bounded Killing field for the model metric $d\varrho^2 + \varrho^2 \gamma + \eta^2$, our fall-off 
hypothesis \eqref{falloff} guarantees that $\nabla_{(a}T_{b)}=O(\varrho^{-2})$ with respect to $h$. Hence
$$0 = \nabla_{(a}\xi_{b)}=  \nabla_{(a}uT_{b)} + O(\varrho^{-2}) = T_{(a} \nabla_{b)}u + O(\varrho^{-2})$$
and we therefore must have $u = \text{const} + O(\varrho^{-1})$. Moreover, the relevant constant
must be non-zero, because $|\xi|= 6|\nabla f|$, and \eqref{gentler} forces $f= \alpha^{-1}$ to be   asymptotically bigger than some positive 
constant times $\varrho$.  
 Thus, some constant multiple $\widehat{\xi}$ of  $\xi$ 
satisfies $\widehat{\xi} = T + O(\varrho^{-1})$. In particular, the  action on 
$\Sigma$  induced  at infinity  by  $\xi$ coincides, up to reparameterization,  with the action of $T$, and so  preserves  $(T,\eta, \gamma)$, as claimed. 
\end{proof}

    \begin{cor} \label{ready}  Let $(M,h)$ be as 
  described by Theorem \ref{judy},  and suppose that that there is an isometric action of the $2$-torus $\mathbb{T}^2$
  on $(M,h)$, such that    the constructed Killing field $\xi$ arises from some element of the Lie algebra $\mathfrak{t}^2$ of $\mathbb{T}^2$.
  Then $(M,h)$ is a non-K\"ahler toric Ricci-flat Hermitian ALF manifold, in the precise technical  sense required  by Biquard and Gauduchon.  
    \end{cor}
  \begin{proof} Lemma \ref{cleared} guarantees that there is an element  $\hat{\xi}\in \mathfrak{t}^2$ whose action  at infinity 
coincides with the flow of the vector field $T$ on $\Sigma$,   exactly  as  
  required by \cite[Definition 1.2]{bg-toric}.     \end{proof}  
  
%

   With these basic facts in hand, we now prove our first rigidity result. 
    
  \begin{thm} 
  \label{baby} Let $(M,h)$ be as 
  described by Theorem \ref{judy}, and suppose that the constructed Killing field $\xi$ is not periodic. 
  Then $(M,h)$ is one of  the  toric  ALF 
   gravitational instantons classified by Biquard and Gauduchon \cite{bg-toric}. 
  \end{thm}

  \begin{proof} The action of the identity component   $\Iso_0 (M,h)$
  of the isometry group preserves the self-dual Weyl curvature $W^+_h$, 
 and  hence  its top eigenvalue $\alpha_h: M \to \RR^+$, and hence  the smooth  proper 
   function $f= \alpha_h^{-1/3}$.  The proof of Lemma \ref{cleared}   shows that $|\nabla f|$ tends to a non-zero constant at infinity, 
 so   the set $X$ of critical points of $f$ is therefore  compact. However, this $X$ is  the zero
   set of both the Killing field $\xi = -6J\nabla f$ and the holomorphic vector field $\xi^{1,0}$. Each connected component
   of the $\Iso_0 (M,h)$-invariant subset $X\subset M$ is therefore either a point or a 
    totally geodesic compact complex curve. Since we also know that 
  $\chi (X) = \chi (M) >  0$, it   follows that some component of $X_0$ of $X$, and hence some orbit $Y\subset X_0$,   is either a point or a $\CP_1$. 
    Equivariance of the exponential map from the normal bundle of $Y$ to $M$  then  guarantees that 
 this gives us   a  faithful representation   $\Iso_0 (M,h)\hookrightarrow \mathbf{U}(2)/\ZZ_\ell$ of the isometry group 
 into some finite quotient of $\mathbf{U}(2)$. 
   
  Let us now consider   the  non-trivial  Killing field $\xi$ as an element of the Lie algebra 
   $\mathfrak{iso}_0 (M,h)$, and  examine the $1$-parameter subgroup 
   $\exp ( \RR \xi)= \{ \exp (t\xi)~|~ t \in \RR\} \subset \Iso_0 (M,h)$ that it determines. The closure 
   of this subgroup
   is  then compact, Abelian, and connected, and so is  a torus $\mathbb{T} =\overline{\exp ( \RR \xi)} \subset\Iso_0 (M,h)\subset \mathbf{U}(2)/\ZZ_\ell$.
  If $\xi$ is periodic, this torus will just be a circle. Otherwise, $\mathbb{T}$   must be a $2$-torus, since $\mathbf{U}(2)$ has rank $2$. 
  In the latter case, the ALF gravitational instanton $(M,h)$ then becomes  Hermitian, non-K\"ahler, and toric, and  so, by Corollary \ref{ready},  falls 
  within  the purview of  the Biquard-Gauduchon classification  \cite[Theorem 8.2]{bg-toric}. 
     \end{proof}

     \begin{cor}
     \label{kindle}
      Let $(M,h_0)$ be a toric Hermitian  ALF gravitational instanton for which the corresponding vector field $T$ on $\Sigma$ 
     is not periodic. Then any Ricci-flat metric $h$ on $M$ which is sufficiently $C^3_1$ close to $h_0$ must be one of the 
      toric  ALF 
   gravitational instantons classified by Biquard-Gauduchon.
     \end{cor}
     \begin{proof} The proof of Lemma \ref{cleared} shows that the vector field $T$ arising from $h_0$ is also the limit
     at infinity of a constant multiple the constructed Killing field $\xi$ of any  $C^3_1$-close Ricci-flat metric $h$. If $T$ is not periodic, it thus 
     follows that the constructed  Killing field $\xi$ of $h$ cannot be periodic, either. The  claim is therefore an immediate consequence of
     Theorem \ref{baby}. 
          \end{proof}

A theorem of    Aksteiner, Andersson, Dahl, Nilsson, and Simon \cite{AADNS}  classifies  those 
  ALF gravitational   instantons that  carry  an isometric $S^1$-action and  are diffeomorphic to either a   Kerr or a Taub-bolt space.  
  Quoting their result  in conjunction with  
  Theorem \ref{judy} 
  would  now allow us to prove a rigidity result in these two cases. 
  However, we will instead  buttress the claims of \cite{AADNS}  
  by   proving this rigidity theorem in a  self-contained way,  by building   directly on the results already obtained in this article.
  
    \setcounter{main}{1}       
  \begin{main}
  \label{bambino}
 Let $(M,h_0)$ be 
  a Kerr or Taub-bolt  gravitational instanton, 
  and let $h$ be another Ricci-flat metric on $M$ that is sufficiently $C^3_1$-close to $h_0$. 
  Then $(M,h)$ is once again a Kerr or Taub-bolt gravitational instanton. 
    \end{main}

\begin{proof} If $(M, h_0)$ is Taub-bolt  or belongs to the Kerr family, then both $(M,h_0)$ and its reverse-oriented version 
$(\overline{M}, h_0)$ are non-K\"ahler and Hermitian toric ALF, and so, by Corollary \ref{ready}, both   fall under the purview of the Biquard-Gauduchon classification. 
In particular, $(M,h_0)$ then satisfies  both $\det (W^+) > 0$ and $\det(W^-) > 0$, and one can therefore apply Theorem \ref{judy}
with respect to either orientation of $M$. If $h$ is a Ricci-flat metric on $M$ that is sufficiently $C^3_1$ close to 
$h_0$, one therefore deduces that $M$ admits  a pair of complex structures $\{ J_+ , J_-\}$ that are respectively compatible with the two 
different orientations of $M$, and  a pair of  extremal K\"ahler metrics $\{ g_+, g_-\}$ that are  compatible with $J_\pm$, respectively,
 where 
$h= \alpha_\pm^{-2} g_\pm$ with  $\alpha_\pm  > 0$   the  top eigenvalues of  $W^\pm_{g_\pm}$, respectively, which is to say  that
$\alpha_\pm^3$ are the  top eigenvalues of $W^\pm_h$. In the terminology 
of \cite{acgI}, the Riemannian manifold $(M,h)$ is therefore  {\sf ambi-K\"ahler}. By  \cite[Proposition 12]{acgI}, it therefore
follows that the Bach-flat manifold $(M,h)$ is at least {\em locally  ambitoric}, and what follows is simply a verification 
that this conclusion actually follows {\em globally} in our case. 

Indeed,  since  the conformal class $[h]$ is Bach-flat, both of the 
 K\"ahler metrics $g_\pm$ must be  extremal, and this then gives rise to 
 two non-trivial Killing fields $J_\pm \nabla^{g_\pm} \alpha_\pm$ on  $(M,h)$. If these Killing fields are linearly independent,
 $(M,h)$ is toric, and we are done.  On the other hand, if they are linearly dependent, one can 
 produce a second Killing field using  case (iii) of the proof of \cite[Proposition 11]{acgI}. Indeed, let $\xi= J_+ \nabla^{g_+} \alpha_+$ 
 be the Killing field associated with $g_+$. After  multiplying $g_-$ by a suitable constant, and replacing $J_-$ with $-J_-$  if necessary, we may 
 then arrange to also have 
 $\xi= J_- \nabla^{g_-} \alpha_-$. Now 
 define
 $\mathscr{S} \in \End( TM )$ by 
 $$
 \mathscr{S}  = \frac{1}{2}  \left(\frac{1}{\alpha_+^2}+ 
 \frac{1}{\alpha_-^2}
 \right) I + \frac{1}{\alpha_+\alpha_-}J_+\circ J_-.
 $$
It then follows from  \cite[Appendix B.5]{acgI}  that $g(\mathscr{S}\cdot , \cdot )$ is a Killing tensor, and that $\mathscr{S} (\xi)$ 
is therefore a Killing field that commutes with $\xi$. 
If $\mathscr{S} (\xi)$ is not the zero field, then 
$(M,h)$ toric. Otherwise, by  \cite[Proposition 12]{acgI},  $(M,g, J_+)$ is 
 the product of two extremal K\"ahler curves, one
 of which has constant curvature, and our asymptotics then guarantee that the latter curve is moreover a round $2$-sphere. 
 In every possible case, $(M,h)$ is  therefore toric, and, in light of Corollary \ref{ready},  the Biquard-Gauduchon classification 
 \cite{bg-toric}  therefore applies. 
The diffeotype of $M$ therefore forces  $(M,h)$ either to  be Taub-bolt, or to  belong  to 
 the   Kerr family. 
\end{proof} 

\noindent
{\bf Added note.} Contemporaneously with the appearance of the first version of this article on the {\sf arXiv}, an e-print by Mingyang Li \cite{mingyang} announced a proof that Hermitian ALF gravitational instantons are always toric. If we take this result for granted, Theorem \ref{bambino} can then 
be improved to just assume that $(M,h_0)$ is a gravitational  instanton appearing in the Biquard-Gauduchon classification,  and  then 
conclude that $(M,h)$ must also belong to the same  family.   Of course, in light of Corollary \ref{kindle}, the gist   of this improvement 
only concerns  the case when $T$ is {\em periodic}. In this periodic  case, Li's proof begins  by compactifying $(M,J)$ as an orbifold
complex surface, and then proceeds to deduce properties of the isometry group of $(M,h)$ from properties of   the complex automorphism group of the compactification.

 \pagebreak 
%

 \bigskip
 
 \noindent {\bf Acknowledgements.} All three authors are grateful for the hospitality of the Institut Mittag-Leffler during its workshop, {\sf Einstein Spaces
 and Special Geometry},  where our collaboration on this project  began. 
  The third author would also like to thank Lars Andersson, Vestislav Apostolov,  Mingyang Li, Song Sun, 
 and Edward Witten for 
 useful discussions of the problem. 

\bigskip

\noindent 
Olivier Biquard\\
{\sc Sorbonne Universit\'e \& Universit\'e  Paris Cit\'e, \\CNRS, IMJ-PRG, F-75005 Paris}\\
{\em Email address:} \href{mailto:olivier.biquard@sorbonne-universite.fr}{olivier.biquard@sorbonne-universite.fr}

\bigskip 

\noindent 
Paul Gauduchon\\
{\sc \'Ecole Polytechniqe, CNRS, CMLS, F-91120 Palaiseau}\\
{\em Email address:} \href{mailto:paul.gauduchon@polytechnique.fr}{paul.gauduchon@polytechnique.fr}

\bigskip 

\noindent 
Claude LeBrun\\ 
{\sc Stony Brook University (SUNY), Stony Brook, NY 11794-3651}\\
{\em Email address:}  \href{mailto:claude@math.stonybrook.edu}{claude@math.stonybrook.edu}  

\end{document}